\documentclass[12pt, reqno]{amsart}
\usepackage{amsmath, amsthm, amscd, amsfonts, amssymb, graphicx, color, float,mathrsfs}
\usepackage[bookmarksnumbered, colorlinks, plainpages]{hyperref}
\hypersetup{colorlinks=true,linkcolor=red, anchorcolor=green, citecolor=cyan, urlcolor=red, filecolor=magenta, pdftoolbar=true}

\theoremstyle{plain}
\newtheorem{theorem}{Theorem}[section]

\newtheorem{lemma}[theorem]{Lemma}
\newtheorem{proposition}[theorem]{Proposition}

\theoremstyle{definition}
\newtheorem{definition}{Definition}[section]

\theoremstyle{remark}
\newtheorem{remark}{Remark}[section]
\newtheorem{example}{Example}[section]

\numberwithin{equation}{section}

\numberwithin{figure}{section} 
\begin{document}
\title[Weighted H-H Inequality in Multiple Variables]
{On the Weighted Hermite-Hadamard Inequality in Multiple Variables, Application for Weighted Multivariate Means}

\author[M. Ra\"{\i}ssouli, L. Tarik, M. Chergui]{Mustapha Ra\"{\i}ssouli$^{1}$, Lahcen Tarik$^2$ and Mohamed Chergui$^{3}$}
\address{$^{1}$ Department of Mathematics, Science Faculty, Moulay Ismail University, Meknes, Morocco.}
\address{$^{2}$ LAGA-Laboratory, Science Faculty,  Ibn Tofail University, Kenitra, Morocco.}
\address{$^{3}$ Department of Mathematics, CRMEF-RSK, EREAM Team, LaREAMI-Lab, Kenitra, Morocco.}

\email{\textcolor[rgb]{0.00,0.00,0.84}{raissouli.mustapha@gmail.com}}
\email{\textcolor[rgb]{0.00,0.00,0.84}{lahcen.tarik@uit.ac.ma}}
\email{\textcolor[rgb]{0.00,0.00,0.84}{chergui\_m@yahoo.fr}}



\date{Received: xxxxxx; Revised: yyyyyy; Accepted: zzzzzz.}

\begin{abstract}
Recently, the so-called Hermite-Hadamard inequality for (operator) convex functions with one variable has known extensive several developments by virtue of its nice properties and various applications. The fundamental target of this paper is to investigate a weighted variant of Hermite-Hadamard inequality in multiple variables that extends the univariate case. As an application, we introduce some weighted multivariate means extending certain bivariate means known in the literature.
\end{abstract}

\maketitle

\section{\bf Introduction}

Let $C$ be a nonempty interval of ${\mathbb R}$. A function $f:C\rightarrow{\mathbb R}$ is called convex if the inequality
\begin{equation}\label{0}
f\big((1-\lambda)a+\lambda b\big)\leq(1-\lambda)f(a)+\lambda f(b)
\end{equation}
holds for any $a,b\in C$ and $\lambda\in[0,1]$. We say that $f$ is concave if \eqref{0} is reversed.

It is well known that $f:C\rightarrow{\mathbb R}$ is convex (resp. concave) if and only if the following inequality
\begin{equation}\label{DJI}
f\Big(\sum_{i=1}^{n+1}\lambda_ia_i\Big)\leq (\geq) \sum_{i=1}^{n+1}\lambda_if(a_i)
\end{equation}
holds true for any $a_1,a_2,...,a_{n+1}\in C$ and $\lambda_1,\lambda_2,...,\lambda_{n+1}\geq 0$ such that $\sum_{i=1}^{n+1}\lambda_i=1$.
\eqref{DJI} is referred to as the discrete Jensen inequality in the literature. An integral version of Jensen inequality reads as follows.
Let $\Omega$ be a $\mu$-measurable set such that $mes(\Omega)>0$. Let $\phi\in
L^1(\Omega)$ be such that $\phi(x)\in C$ almost everywhere and
$f\circ\phi\in L^1(\Omega)$. If $f:C\rightarrow{\mathbb R}$ is convex then
\begin{equation}\label{IJI}
f\left(\displaystyle\frac{1}{mes(\Omega)}\int_{\Omega}\phi(x)d\mu(x)\right)\leq
\frac{1}{mes(\Omega)}\int_{\Omega}f\big(\phi(x)\big)d\mu(x).
\end{equation}
If $f:C\rightarrow{\mathbb R}$ is concave, the inequalities in \eqref{IJI} are reversed.

As an application of \eqref{0} and \eqref{IJI}, we can easily deduce the following double inequality
\begin{equation}\label{1}
f\Big(\dfrac{a+b}{2}\Big)\leq\int_0^1f\Big((1-t)a+tb\Big)dt\leq\dfrac{f(a)+f(b)}{2},
\end{equation}
whenever $a,b\in C$ and $f:C\rightarrow{\mathbb R}$ is convex. If $f$ is concave, the inequalities in \eqref{1} are reversed. Inequalities \eqref{1}, known in the literature as the Hermite-Hadamard inequality $(HHI)$, are useful in mathematical analysis and contributes as a good tool for determining interesting estimations and approximations. An enormous amount of efforts has been devoted in the literature for extending \eqref{1} from the case where the variables are real numbers to the case where the variables are bounded linear operators. For more extensions as well as some refinements, reverses and applications of \eqref{1}, we refer the interested reader to \cite{BT,D2,D3,D4,M,RAI1,RAI2} and the related references cited therein.

$(HHI)$ have been extended from the uni-dimensional case to the multidimensional case, as explained in what follows, \cite{BES,NIC}. Let $D_n$ be a $(n+1)$-simplex of ${\mathbb R}^n$ and $f:D_n\longrightarrow{\mathbb R}$ be a convex function. If $p_1,p_2,...,p_{n+1}$ denote the vertices of $D_n$, then we have
\begin{equation}\label{13}
f\Big(\sum_{i=1}^{n+1}\frac{p_i}{n+1}\Big)\leq\frac{1}{|D_n|}
\int_{D_n}f(x)dx\leq\frac{1}{n+1}\sum_{i=1}^{n+1}f(p_i),
\end{equation}
where $x:=(x_1,x_2,...,x_n)\in D_n$, $dx:=dx_1dx_2...dx_n$ and $|D_n|:=\int_{D_n}dx$ stands for the Lebesgue volume of $D_n$ in ${\mathbb R}^n$.

Some refinements of \eqref{13} can be found in \cite{MSP,NW,NW1,NW2,PA,RAD} and for a reverse version of \eqref{13} one can consult \cite{MS}. An extension of \eqref{13}, due to Choquet, for convex functions on a compact set was investigated in \cite{PH}. Let $E_n$ be the standard simplex of ${\mathbb R}^n$ i.e.,
$$E_n:=\Big\{(t_1,t_2,...,t_n)\in{\mathbb R}^n:\; t_i\geq0,\; i=1,2,...,n,\; \sum_{i=1}^nt_i\leq1\Big\}.$$
Then $|E_n|=1/n!$ the Lebesgue volume of $E_n$. From \eqref{13} we infer that
\begin{equation}\label{14}
f\Big(\sum_{i=1}^{n+1}\frac{a_i}{n+1}\Big)\leq
n!\int_{E_n}f\Big(\sum_{i=1}^{n+1}t_ia_i\Big)dt\leq\frac{1}{n+1}\sum_{i=1}^{n+1}f(a_i),
\end{equation}
provided that $f$ is convex on a convex set containing $a_1,a_2,...,a_{n+1}$. Here and throughout the following, we set $dt:=dt_1dt_2...dt_n$ and $t_{n+1}:=1-\sum_{i=1}^nt_i$. If $f$ is concave then the inequalities in \eqref{13} and in \eqref{14} are reversed.

The remainder of this paper will be organized as follows: Section 2 deals with some basic notions and results that will be needed throughout the following. In Section 3, we investigate in two versions a weighted Hermite-Hadamard inequality which is an extension of \eqref{14}. Section 4 deals with the application of our previous theoretical results to the introduction of certain weighted means in multiple variables.

\section{\bf Some needed tools and results}

To provide our main results, we need more basic notions. Let $x>0$ and $k>0$ be given. The notations $\Gamma$ and $(x)_k$ refer respectively to the special gamma function  and the Pochhammer symbol defined by
$$\Gamma(x):=\int_0^{\infty}t^{x-1}e^{-t}dt,\;\; (x)_k:=\frac{\Gamma(x+k)}{\Gamma(x)}.$$

Recall that we have $\Gamma(x+1)=x\Gamma(x)$ for all $x>0$, with $\Gamma(1)=1$, and so $\Gamma(n+1)=n!$ for any integer $n\geq1$. The following lemma will be needed later.

\begin{lemma}
Let $k>0$ be an integer and $c>0$ be a real number. The equation in $x>0$: $(x)_k=c$ has one and only one root which will be denoted by $x=[c]_{-k}$.
\end{lemma}
\begin{proof}
For fixed $k,c>0$, we put $p(x)=(x)_k-c=x(x+1)...(x+k-1)-c$. It is clear that $p(0):=\lim_{x\downarrow0}p(x)=-c<0$ and $p(+\infty)=+\infty$. By the standard mean value theorem in Real Analysis there is some $r>0$ such that $p(r)=0$. It is obvious that $x\mapsto p(x)$ is strictly increasing and therefore the root $r$ is unique, so completing the proof.
\end{proof}

Some basic properties of $x\mapsto[x]_{-k}$, which will be needed in the sequel, are embodied in the following result.

\begin{proposition}\label{pr1}
Let $x>0$ be a real number and $k>0$ be an integer.We have the following assertions:\\
(i) $\big([x]_{-k}\big)_k=x$ and $\big[(x)_k\big]_{-k}=x$.\\
(ii) $[x]_{-1}=x$.\\
(iii) $[k!]_{-k}=1$.
\end{proposition}
\begin{proof}
It is straightforward and we omit the details here.
\end{proof}

 We also need to recall the following. The beta function in multiple variables is defined for $x_1,x_2,...,x_n,x_{n+1}>0$ by, see \cite{ALZ,RMSN} for instance,
\begin{equation}\label{32}
B_{n+1}\big(x_1,x_2,...,x_n,x_{n+1}\big):=\int_{E_{n}}\prod_{i=1}^{n+1}t_i^{x_i-1}dt_1dt_2...dt_{n},
\end{equation}
where, as previous, we set $t_{n+1}=1-\sum_{i=1}^{n}t_i$. For $n=1$, \eqref{32} defines exactly the standard beta function in two variables.

We have the following formula which states a relationship between this beta function in multiple variables and the standard gamma function,

\begin{equation}\label{33}
B_{n+1}\big(x_1,x_2,...,x_n,x_{n+1}\big)=\frac{\prod_{i=1}^{n+1}\Gamma(x_i)}{\Gamma\big(\sum_{i=1}^{n+1}x_i\big)},
\end{equation}
and so $B_{n+1}\big(x_1,x_2,...,x_n,x_{n+1}\big)$ is symmetric in $x_1,x_2,...,x_{n+1}$.

We also denote by $int(E_n)$ the topological interior of $E_n$ in ${\mathbb R}^n$, namely
$$int(E_n):=\Big\{(t_1,t_2,...,t_n)\in{\mathbb R}^n:\; t_i>0,\; i=1,2,...,n,\; \sum_{i=1}^nt_i<1\Big\}.$$

We have the following lemma.

\begin{lemma}\label{lem1}
For $\lambda_{n+1}:=1-\sum_{i=1}^n\lambda_i$ and ${\bf\lambda}:=(\lambda_1,\lambda_2,...,\lambda_n)\in int(E_n)$ there holds
$$\displaystyle\int_{E_n}\sum_{i=1}^{n+1}\lambda_it_i^{\lambda_{i,n}}\;dt_1dt_2...dt_n=\frac{1}{n!},$$
where, for $i=1,2,...,n,n+1$, we set
\begin{equation}
\lambda_{i,n}:=\Big[\frac{n!\lambda_i}{\lambda_{i+1}}\Big]_{-n}-1 \;\; \mbox{and}\;\; \lambda_{n+2}:=\lambda_1.
\end{equation}
\end{lemma}
\begin{proof}
By \eqref{32} and the fact that $B_{n+1}\big(x_1,x_2,...,x_n,x_{n+1}\big)$ is symmetric in $x_1,x_2,...,x_{n+1}$ we have, by setting $dt:=dt_1dt_2...dt_n$,
$$\displaystyle\int_{E_n}t_i^{\lambda_{i,n}}\;dt=B_{n+1}\big(\lambda_{i,n}+1,1,1,...,1\big).$$
This, with \eqref{33} and the definition of $(x)_k$, yields
$$\displaystyle\int_{E_n}t_i^{\lambda_{i,n}}\;dt=\frac{\Gamma\big(\lambda_{i,n}+1\big)}{\Gamma\big(\lambda_{i,n}+n+1\big)}=\frac{1}{\big(\lambda_{i,n}+1\big)_n}.$$
Using the implication
\begin{equation}\label{IMP}
\lambda_{i,n}:=\Big[\frac{n!\lambda_i}{\lambda_{i+1}}\Big]_{-n}-1\;\Longrightarrow\;\big(\lambda_{i,n}+1\big)_n=\frac{n!\lambda_i}{\lambda_{i+1}}
\end{equation}
we then get
$$\displaystyle\int_{E_n}t_i^{\lambda_{i,n}}\;dt=\frac{\lambda_{i+1}}{n!\lambda_i}.$$
Multiplying this latter equality by $\lambda_i$ for $i=1,2,...,n+1$ and summing from $i=1$ to $i=n+1$, we obtain
$$\sum_{i=1}^{n+1}\lambda_i\displaystyle\int_{E_n}t_i^{\lambda_{i,n}}\;dt=\sum_{i=1}^{n+1}\frac{\lambda_{i+1}}{n!}=\frac{1}{n!},$$
since $\sum_{i=1}^{n+1}\lambda_i=1$ and $\lambda_{n+2}=\lambda_1$. The proof is finished.
\end{proof}

The following lemma will be also needed.

\begin{lemma}\label{lem2}
Let $i,j=1,2,...,n+1$. Then there holds
$$\int_{E_n}t_jt_i^{\lambda_{i,n}}dt=
\left\{
\begin{tabular}{lll}
$\dfrac{\lambda_{i+1}}{n!\lambda_i}\dfrac{\lambda_{i,n}+1}{\lambda_{i,n}+n+1}\;\;\;\; \mbox{if}\;\; j=i$\\\\
$\dfrac{\lambda_{i+1}}{n!\lambda_i}\dfrac{1}{\lambda_{i,n}+n+1}\;\;\;\; \mbox{if}\;\; j\neq i,$
\end{tabular}
\right.$$
where, as above, $dt:=dt_1dt_2...dt_n$ and $\lambda_{n+2}=\lambda_1$.
\end{lemma}
\begin{proof}
Assume that $j=i$. By \eqref{32} and the fact that $B_{n+1}\big(x_1,x_2,...,x_n,x_{n+1}\big)$ is symmetric in $x_1,x_2,...,x_{n+1}$, we get
$$\int_{E_n}t_jt_i^{\lambda_{i,n}}dt=\int_{E_n}t_i^{\lambda_{i,n}+1}dt=B_{n+1}\big(\lambda_{i,n}+2,1,1,...,1\big),$$
which, with \eqref{33} and the relationship $\Gamma(x+1)=x\Gamma(x)$, yields
\begin{multline*}
\int_{E_n}t_jt_i^{\lambda_{i,n}}dt=\frac{\Gamma(\lambda_{i,n}+2)}{\Gamma(\lambda_{i,n}+n+2)}
=\frac{(\lambda_{i,n}+1)\;\Gamma(\lambda_{i,n}+1)}{(\lambda_{i,n}+n+1)\;\Gamma(\lambda_{i,n}+n+1)}\\
=\frac{\lambda_{i,n}+1}{\lambda_{i,n}+n+1}\frac{1}{(\lambda_{i,n}+1)_n}.
\end{multline*}
Thanks to \eqref{IMP}, we get the desired result when $j=i$. Now, assume that $j\neq i$. By the same arguments as previous, and using again the fact that $B_{n+1}$ is symmetric in its variables and $\Gamma(2)=1$, we obtain
\begin{multline*}
\int_{E_n}t_jt_i^{\lambda_{i,n}}dt=B_{n+1}\big(\lambda_{i,n}+1,2,1,...,1\big)
=\frac{\Gamma(\lambda_{i,n}+1)}{\Gamma(\lambda_{i,n}+n+2)}\\
= \frac{\Gamma(\lambda_{i,n}+1)}{(\lambda_{i,n}+n+1)\;\Gamma(\lambda_{i,n}+n+1)}=\frac{1}{\lambda_{i,n}+n+1}\frac{1}{(\lambda_{i,n}+1)_n},
\end{multline*}
which with \eqref{IMP} again implies the desired result, so completing the proof.
\end{proof}

\section{\bf Weighted Hermite-Hadamard inequalities}

We preserve the same notations as previous. For ${\bf\lambda}:=(\lambda_1,\lambda_2,...,\lambda_n)\in E_n$ and $a:=(a_1,a_2,...,a_n,a_{n+1})\in{\mathbb R}^{n+1}$ we set
$$\nabla_\lambda a:=\sum_{i=1}^{n+1}\lambda_ia_i,\; \mbox{with}\; \lambda_{n+1}:=1-\sum_{i=1}^n\lambda_i.$$
With this, if $f:C\rightarrow{\mathbb R}$ is convex then \eqref{DJI} can be shortly written as follows
\begin{equation}\label{S}
f\Big(\nabla_\lambda a\Big)\leq\nabla_\lambda f(a),
\end{equation}
where $a:=(a_1,a_2,...,a_n,a_{n+1})\in C^{n+1}$ and $f(a):=\big(f(a_1),f(a_2),...,f(a_{n+1})\big)$.

As already pointed out before, we will present in this section a weighted Hermite-Hadamard inequalities that extend \eqref{14}. Two versions of these inequalities will be displayed here. The first version reads as follows.

\begin{theorem}\label{thA}
Let $C$ be a nonempty convex set of ${\mathbb R}^n$ and $f:C\rightarrow{\mathbb R}$ be convex. If ${\bf\lambda}:=(\lambda_1,\lambda_2,...,\lambda_n)\in int(E_n)$ then for any $a:=(a_1,a_2,...,a_n,a_{n+1})\in C^{n+1}$ there holds
\begin{equation}\label{50}
f\Big(\sum_{i=1}^{n+1}\tilde{\lambda}(i,n)a_i\Big)\leq\displaystyle\int_{E_n}f\left(\sum_{i=1}^{n+1}t_ia_i\right)d\nu_\lambda(t)
\leq\sum_{i=1}^{n+1}\tilde{\lambda}(i,n)f(a_i),
\end{equation}
where $t_{n+1}:=1-\sum_{i=1}^nt_i$, $dt:=dt_1dt_2...dt_n$, $\lambda_{n+1}:=1-\sum_{i=1}^n\lambda_i$, $\nu_{\bf \lambda}(t)$ is the probability measure on  $E_n$ defined through
\begin{equation}\label{nu}
d\nu_{\lambda}(t):=n!\sum_{i=1}^{n+1}\lambda_it_i^{\lambda_{i,n}}\;dt,\;\; \lambda_{i,n}:=\Big[\frac{n!\lambda_i}{\lambda_{i+1}}\Big]_{-n}-1, \;\; \mbox{and}\;\; \lambda_{n+2}:=\lambda_1,
\end{equation}
and $\tilde{\lambda}(i,n)$ is defined by
\begin{equation}\label{tilde}
\tilde{\lambda}(i,n):=\sum_{j=1,j\neq i}^{n+1}\frac{\lambda_{j+1}}{\lambda_{j,n}+n+1}+\lambda_{i+1}\frac{\lambda_{i,n}+1}{\lambda_{i,n}+n+1}.
\end{equation}
If $f$ is concave, the inequalities in \eqref{50} are reversed.
\end{theorem}
\begin{proof}
Lemma \ref{lem1} asserts that $\nu_{\bf \lambda}(t)$ is really a probability measure on  $E_n$. With this, the discrete and integral Jensen inequalities yield
\begin{equation}\label{A}
f\left(\sum_{j=1}^{n+1}a_j\displaystyle\int_{E_n}t_jd\nu_{\lambda}(t)\right)\leq\displaystyle\int_{E_n}f\left(\sum_{j=1}^{n+1}t_ja_j\right)d\nu_\lambda(t)
\leq\sum_{j=1}^{n+1}f(a_j)\displaystyle\int_{E_n}t_jd\nu_{\lambda}(t).
\end{equation}
If we replace $d\nu_\lambda(t)$ by its expression given in \eqref{nu}, we get
$$\displaystyle\int_{E_n}t_jd\nu_{\lambda}(t)=n!\sum_{i=1}^{n+1}\lambda_i\int_{E_n}t_jt_i^{\lambda_{i,n}}dt,$$
which gives by using Lemma \ref{lem2} and making some simple simplifications,
$$\displaystyle\int_{E_n}t_jd\nu_{\lambda}(t)=\sum_{i=1,i\neq j}^{n+1}\frac{\lambda_{i+1}}{\lambda_{i,n}+n+1}+\lambda_{j+1}\frac{\lambda_{j,n}+1}{\lambda_{j,n}+n+1}:=\tilde{\lambda}(j,n).$$
Substituting these in \eqref{A} we get \eqref{50}, so completing the proof.
\end{proof}

Now, let us state the following remark which may be of interest for the reader.

\begin{remark}\label{rem2} (i) One can check that $\sum_{i=1}^{n+1}\tilde{\lambda}(i,n)=1$ for any integer $n \ge 1$.\\
(ii) If $n=1$ and $\lambda=1/2$ then $d\nu_\lambda(t)=dt$ and \eqref{50} is reduced to \eqref{1}.\\
(iii) If $n=1$ then for $i=1,2$ we have, by using the definition of $\lambda_{i,n}$ and Proposition \ref{pr1},(ii),
$$\lambda_{i,1}:=\Big[\frac{\lambda_i}{\lambda_{i+1}}\Big]_{-1}-1=\frac{\lambda_i}{\lambda_{i+1}}-1.$$
This, with $\lambda_1=\lambda$ and $\lambda_2=1-\lambda$ and the fact that $\lambda_3=\lambda_1$, gives
$$\lambda_{1,1}=\frac{\lambda}{1-\lambda}-1=\frac{2\lambda-1}{1-\lambda},\; \lambda_{2,1}=\frac{1-\lambda}{\lambda}-1=\frac{1-2\lambda}{\lambda},$$
which with \eqref{tilde} and some elementary calculations yields $\tilde{\lambda}(1,1)=\lambda$ and $\tilde{\lambda}(2,1)=1-\lambda$.
Thus, for $n=1$ and $\lambda\in]0,1[$, \eqref{50} can be written as follows
\begin{equation*}
f\Big(\lambda a+(1-\lambda)b\Big)\leq\int_0^1f\Big(ta+(1-t)b\Big)d\nu_\lambda(t)\leq\lambda f(a)+(1-\lambda)f(b),
\end{equation*}
where $\nu_\lambda$ is the probability measure defined on $[0,1]$ by
\begin{equation*}
d\nu_\lambda(t)=\Big(\lambda t^{\frac{2\lambda-1}{1-\lambda}}+(1-\lambda)(1-t)^{\frac{1-2\lambda}{\lambda}}\Big)dt.
\end{equation*}
(iv) If we choose $\lambda_i=\frac{1}{n+1}$ for all $i=1,2,...,n+1$ then by Proposition \ref{pr1},(iii) we have $\lambda_{i,n}=[n!]_{-n}-1=0$ and by \eqref{tilde} we get $\tilde{\lambda}(i,n)=\frac{1}{n+1}$ for any $i=1,2,...,n+1$. It follows that $d\nu_\lambda(t)=n!\;dt$ and therefore the inequalities \eqref{50} coincide with \eqref{14}.
\end{remark}

Now, we wish to give a weighted Hermite-Hadamard inequalities where the left bound and the right bound are, respectively, the two sides involved in \eqref{S}. We need to state the following lemma.

\begin{lemma}\label{lem3}
Let $\alpha:=(\alpha_1,\alpha_2,...,\alpha_n)\in{\mathbb R}^{n}$ with $\alpha_i>0$, for $i=1,2,...,n$. Consider the $\alpha$-power simplex of ${\mathbb R}^n$ defined by
$$F_n:=\Big\{(t_1,t_2,...,t_n)\in{\mathbb R}^n:\;t_i\geq0, i=1,2,..,n,\; t_1^{\alpha_1}+t_2^{\alpha_2}+...+t_n^{\alpha_n}\leq1\Big\}.$$
The Lebesgue Volume of $F_n$ is given by
\begin{equation}\label{Vol}
|F_n|:=\int_{F_n}\;dt=\frac{B_n\big(\alpha_1^{-1},\alpha_2^{-1},...,\alpha_n^{-1}\big)}{\Big(\prod_{i=1}^n\alpha_i\Big)\Big(\sum_{i=1}^n\alpha_i^{-1}\Big)}.
\end{equation}
In another part, we have
\begin{equation}\label{vol}
\frac{1}{|F_n|}\int_{F_n}t_i^{\alpha_i}\;dt=\frac{\alpha_i^{-1}}{\sum_{i=1}^n\alpha_i^{-1}+1},\;\; i=1,2,...,n.
\end{equation}
\end{lemma}
\begin{proof}
Making the change of variables $u_i=t_i^{\alpha_i}$ for $i=1,2,...,n$ and computing the corresponding Jacobian, we get
$$|F_n|=\frac{1}{\prod_{i=1}^n\alpha_i}\int_{E_n}\prod_{i=1}^nu_i^{\alpha_i^{-1}-1}\;du,$$
where $du:=du_1du_2...du_n$. This, with \eqref{32}, yields
$$|F_n|=\frac{1}{\prod_{i=1}^n\alpha_i}B_{n+1}\big(\alpha_1^{-1},\alpha_2^{-1},...,\alpha_n^{-1},1\big),$$
which, with \eqref{33} and the relationship $\Gamma(x+1)=x\;\Gamma(x)$, implies \eqref{Vol}. To establish \eqref{vol} we use the same tools and we left the details here to the reader.
\end{proof}

Now, we are in the position to state the following result which provides an answer to our previous claim.

\begin{theorem}
With the same notations and hypotheses as in Theorem \ref{thA}, the following inequalities hold
\begin{equation}\label{55}
f\Big(\nabla_{\bf \lambda} a\Big)\leq\displaystyle\int_{E_n}f\left(\sum_{i=1}^{n+1}t_ia_i\right)d\mu_\lambda(t)\leq\nabla_{\bf \lambda} f(a),
\end{equation}
where $\mu_{\bf \lambda}(t)$ is the probability measure on  $E_n$ defined through
\begin{equation}\label{mu}
d\mu_{\lambda}(t):=\frac{\sum_{i=1}^n\alpha_i^{-1}}{B_n\big(\alpha_1^{-1},...,\alpha_n^{-1}\big)}
\prod_{i=1}^{n}t_i^{\alpha_i^{-1}-1}\;dt,\; \mbox{with}\; \alpha_i:=\frac{\lambda_{n+1}}{\lambda_{i}}.
\end{equation}
If $f$ is concave, the inequalities in \eqref{55} are reversed.
\end{theorem}
\begin{proof}
Let $F_n$ be as defined in Lemma \ref{lem3}. By the discrete and integral Jensen inequalities we get
\begin{equation}\label{F}
f\left(\frac{1}{|F_n|}\sum_{i=1}^{n+1}a_i\int_{F_n}t_i^{\alpha_i}\;dt\right)\leq\frac{1}{|F_n|}\int_{F_n}f\Big(\sum_{i=1}^{n+1}t_i^{\alpha_i}a_i\Big)dt
\leq\frac{1}{|F_n|}\sum_{i=1}^{n+1}f(a_i)\int_{F_n}t_i^{\alpha_i}\;dt,
\end{equation}
where $dt:=dt_1dt_2...dt_n$\, and\, $t_{n+1}^{\alpha_{n+1}}:=1-\sum_{i=1}^nt_i^{\alpha_i}$.

Now, by \eqref{vol} with the fact that $\alpha_i:=\lambda_{n+1}/\lambda_i$ we get
$$\frac{1}{|F_n|}\int_{F_n}t_i^{\alpha_i}\;dt=\frac{\alpha_i^{-1}}{\sum_{i=1}^n\alpha_i^{-1}+1}=\frac{\lambda_{n+1}^{-1}\lambda_i}
{\sum_{i=1}^n\lambda_{n+1}^{-1}\lambda_i+1}=\frac{\lambda_i}{\sum_{i=1}^n\lambda_i+\lambda_{n+1}}=\lambda_i,$$
which when substituted in the extreme sides of \eqref{F} yields the two extremes sides of \eqref{55}.

Finally, by the same change of variables used in Lemma \ref{lem3} and using \eqref{Vol} and \eqref{mu}, we can check that
$$\frac{1}{|F_n|}\int_{F_n}f\Big(\sum_{i=1}^{n+1}t_i^{\alpha_i}a_i\Big)dt=\displaystyle\int_{E_n}f\left(\sum_{i=1}^{n+1}t_ia_i\right)d\mu_\lambda(t),$$
which concludes the proof.
\end{proof}

We end this section by stating the following remark which explains that \eqref{55} is a generalization of \eqref{1} and \eqref{14}.

\begin{remark}\label{rem3}
(i) If $n=1$ and $\lambda=1/2$ then $d\nu_\lambda(t)=dt$ and \eqref{55} is reduced to \eqref{1}.\\
(ii) If $n=1$ then simple computations lead to
$$\alpha_1=\frac{\lambda}{1-\lambda},\;d\mu_{\lambda}(t)=\frac{1-\lambda}{\lambda}t^{\frac{1-2\lambda}{\lambda}}\;dt.$$
So, for $n=1$ and $\lambda\in]0,1[$, \eqref{55} becomes
\begin{equation*}
f\Big(\lambda a+(1-\lambda)b\Big)\leq\int_0^1f\Big(ta+(1-t)b\Big)d\mu_\lambda(t)\leq\lambda f(a)+(1-\lambda)f(b),
\end{equation*}
(iii) If we choose $\lambda_i=\frac{1}{n+1}$ for all $i=1,2,...,n+1$ then $\alpha_i=1$ for all $i=1,2,...,n$. By \eqref{mu}, we infer that $d\mu_{\lambda}(t)=n!\;dt$ and so \eqref{55} coincides with \eqref{14}.
\end{remark}

\section{\bf Application for weighted multivariate means}

As we have already pointed out, this section focuses on an application of the previous weighted Hermite-Hadamard inequalities to construct weighted multivariate means. We keep the same notations as in the previous section and, for the sake of simplicity, we set $a:=(a_1,a_2,...,a_{n+1})$, with $a_1>0, a_2>0,..., a_{n+1}>0$. We also put $\lambda:=(\lambda_1,\lambda_2,...,\lambda_n)\in E_n$, with $\lambda_{n+1}:=1-\sum_{i=1}^n\lambda_i$. The notation $(e_1,e_2,...,e_{n})$ refers to the canonical basis of ${\mathbb R}^{n}$ and we put $e:=\frac{1}{n+1}(1,1,...,1)\in{\mathbb R}^{n}$. The zero-vector of ${\mathbb R}^n$ will be denoted by ${\bf 0}$.

\subsection{\bf Background about some standard means.} By multivariate mean $m$ \cite{MAT}, we understand a map in $(n+1)$-variables $a_1,a_2,...,a_n,a_{n+1}>0$ such that
\begin{equation}\label{mean}
\min(a_1,a_2,...,a_{n+1})\leq m(a_1,a_2,...,a_{n+1})\leq\max(a_1,a_2,...,a_{n+1}).
\end{equation}
Symmetric and homogeneous multivariate means are defined in the habitual way. Assuming the existence of the first partial derivatives of $m$ at the point $\left(1,1, \ldots, 1\right)$. We say that $m$ is a weighted mean if
\begin{definition}\label{DP}
With the previous notation, let $m_\lambda$ be a $(n+1)$-variables map indexed by $\lambda:=(\lambda_1,\lambda_2,...,\lambda_n)\in E_n$. We say that $m_\lambda$ is a weighted multivariate mean if the following requirements are satisfied:\\
(i) $m_\lambda$ is a (multivariate) mean in the sense of \eqref{mean}, for any $\lambda:=(\lambda_1,\lambda_2,...,\lambda_n)\in E_n$,\\
(ii) $m_{e_i}(a_1,a_2,...,a_{n+1})=a_i$ for any $i=1,2,...,n$, and $m_{\bf 0}(a_1,a_2,...,a_{n+1})=a_{n+1}$,\\
(iii) $m_e(a_1,a_2,...,a_{n+1})$ is symmetric in $a_1,a_2,..,a_{n+1}$,\\
(iv) The first partial derivatives of the multivariate map $(a_1,a_2,...,a_{n+1})\longmapsto m_\lambda(a_1,a_2,...,a_{n+1})$ exist at the point $\left(1,1, \ldots, 1\right)\in{\mathbb R}^{n+1}$, and we have
$$\lambda_i=\frac{\partial m_\lambda}{\partial a_i}(1,1, \ldots, 1), \quad \text{for all }\; i=1,2, \ldots, n+1.$$

\end{definition}
If $m_\lambda$ is symmetric in $a_1, \ldots, a_{n+1}$, the weights $\lambda_i$ are independent of $i$, leading to $\lambda_i=\frac{1}{n+1}$, i.e. $m_\lambda= m_e$. In this case, $m_e$ is called the associated symmetric mean of $m_\lambda$ and $m_\lambda$ is called the weighted (multivariate) $m_e$-mean.

The standard weighted multivariate means are defined as follows
\begin{equation*}
\nabla_\lambda a:=\sum_{i=1}^{n+1}\lambda_ia_i,\; !_\lambda a:=\Big(\sum_{i=1}^{n+1}\lambda_ia_i^{-1}\Big)^{-1}, \;
\sharp_\lambda a:=\prod_{i=1}^{n+1}a_i^{\lambda_i}.
\end{equation*}
They are known in the literature as the weighted arithmetic mean, the weighted harmonic mean, and the weighted geometric mean, respectively. We can easily check that they satisfy the condition of Definition \ref{DP} and their associated symmetric means are, respectively, given by
$$\nabla a:=\frac{1}{n+1}\sum_{i=1}^{n+1}a_i,\; !a:=(n+1)\Big(\sum_{i=1}^{n+1}a_i^{-1}\Big)^{-1},\; \sharp a:=\sqrt[n+1]{\prod_{i=1}^{n+1}a_i}.$$
The following inequalities $!_\lambda a\leq \sharp_\lambda a\leq \nabla_\lambda a$ hold, and play an important role in mathematical inequalities as well as in applied science such as in probabilities and statistics. For $n=1$, they are the (bivariate) weighted means defined, for $a, b>0$ and $\lambda\in[0,1]$, by
$$a\nabla_\lambda b:=(1-\lambda)a+\lambda b,\; a!_\lambda b:=\Big((1-\lambda)a^{-1}+\lambda b^{-1}\Big)^{-1},\; a\sharp_\lambda b:=a^{1-\lambda}b^{\lambda},$$
respectively. Their associated symmetric means correspond to the case where $\lambda=1/2$ and are, respectively, given by
$$a\nabla b=\frac{a+b}{2}, a!b=\dfrac{2ab}{a+b},\; a\sharp b=\sqrt{ab}.$$
Two other weighted means in two variables $a,b>0$ have been introduced in the literature \cite{PSMA}, namely
\begin{equation}\label{L}
L_\lambda(a,b):=\dfrac{1}{\log a-\log b}\left(\dfrac{1-\lambda}{\lambda}\big(a-a^{1-\lambda}b^\lambda\big)+\dfrac{\lambda}{1-\lambda}\big(a^{1-\lambda}b^\lambda-b\big)\right),\; L_\lambda(a,a)=a
\end{equation}
\begin{equation}\label{I}
I_\lambda(a,b):=\frac{1}{e}\Big(a\nabla_\lambda b\Big)^{\frac{(1-2\lambda)(a\nabla_\lambda b)}
{\lambda(1-\lambda)(b-a)}}\left(\dfrac{b^{\frac{\lambda b}{1-\lambda}}}{a^{\frac{(1-\lambda)a}{\lambda}}}\right)^{\frac{1}{b-a}},\; I_\lambda(a,a)=a,
\end{equation}
with $L_0(a,b):=\lim\limits_{\lambda\downarrow0}L_\lambda(a,b)=a$ and $L_1(a,b):=\lim\limits_{\lambda\uparrow1}L_\lambda(a,b)=b$, and similar equalities for $I_\lambda(a,b)$. Their associated symmetric means (i.e. corresponding to $\lambda=1/2$) are the known logarithmic mean and identric mean, respectively, defined by:
$$L(a,b):=\frac{b-a}{\log b-\log a},\; L(a,a)=a;\;\; I(a,b):=\frac{1}{e}\Big(\frac{b^b}{a^a}\Big)^{\frac{1}{b-a}},\; I(a,a)=a.$$
Then $L_\lambda(a,b)$ and $I_\lambda(a,b)$ are called the weighted logarithmic mean and the weighted identric mean, respectively. The following inequalities hold, \cite{PSMA,RAF}
\begin{equation*}
a\sharp_\lambda b\leq L_\lambda(a,b)\leq a\nabla_\lambda b,\;\;
a\sharp_\lambda b\leq I_\lambda(a,b)\leq a\nabla_\lambda b.
\end{equation*}

We have also the following chain of inequalities, see \cite{DP} for instance
\begin{equation*}
a!b\leq a\sharp b\leq L(a,b)\leq I(a,b)\leq a\nabla b.
\end{equation*}

\begin{remark}\label{rem1}
(i) When a weighted multivariate mean $m_\lambda$ is given, its associated symmetric mean $m_{e}$ is of course unique. However, it is possible to have more than one weighted multivariate mean whose the associated symmetric mean is the same. For more details, see the subsections below. See also \cite{RCT} where the authors introduced two weighted logarithmic means and one weighted identric mean that are mutually different from \eqref{L} and \eqref{I}, respectively.\\
(ii) For another short writing of \eqref{L} and \eqref{I}, with more properties and applications, we can consult \cite{RAF}.
\end{remark}

Given the aforementioned, it is normal to pose the following question:

{\it Question:} In multiple variables, how can one introduce a weighted logarithmic mean and a weighted identric mean?

The weighted Hermite-Hadamard inequality investigated previously will be a good tool for answering this latter question. We itemize the details in the following subsections.

\subsection{\bf Weighted multivariate logarithmic mean.}

As already pointed out, our aim here is to provide a new weighted logarithmic mean in multiple variables. We begin by stating the following proposition.

\begin{proposition}
Let ${\bf\lambda}:=(\lambda_1,\lambda_2,...,\lambda_n)\in int(E_n)$ and $a:=(a_1,a_2,...,a_n,a_{n+1})\in (0,\infty)^{n+1}$. We set
\begin{equation}\label{L1}
\mathcal{L}_\lambda(a):=\int_{E_{n}}  \sharp_{t}a\, d \mu_\lambda(t),
\end{equation}
\begin{equation}\label{L2}
{\mathbb L}_\lambda(a):=\left(\int_{E_n}\,\left(\nabla_t\,a\right)^{-1} d\mu_\lambda(t)\right)^{-1},
\end{equation}
with $t_{n+1}=1-\sum_{i=1}^{n}t_i$ and $\mu_\lambda$ is the measure defined in \eqref{mu}.
Then ${\mathcal L}_\lambda$ and ${\mathbb L}_\lambda$ are multivariate weighted means, that we call weighted logarithmic means.
\end{proposition}
\begin{proof}
We know that $a\longmapsto\sharp_t a$ is a mean for any fixed $t\in E_n$, that is
$$\min\left(a_1,a_2,...,a_{n+1}\right)\le \sharp_t a\le \max\left(a_1,a_2,...,a_{n+1}\right).$$
Multiplying all sides of this latter double inequality by $d\mu_\lambda(t)$ and then integrating with respect to $(t_1,t_2,..., t_n)\in E_n$, we get that $\mathcal{L}_\lambda(a)$ is mean for any $\lambda$. For condition (ii) of Definition \ref{DP}, it is a direct consequence of inequalities \eqref{55}. From \eqref{L1}, with Remark \ref{rem3},(iii), it is easy to see that
\begin{equation*}
\displaystyle\mathcal{L}_e(a)=n!\;\int_{E_n}\sharp_t a\,dt,
\end{equation*}
which with the fact that, 
$$(t_1,t_2,...,t_{n+1})\in E_{n}\Longleftrightarrow \big(t_{\sigma(1)},t_{\sigma(2)},...,t_{\sigma(n+1)}\big)\in E_{n}$$
for any permutation $\sigma$ of $\{1,2,...,n+1\}$, implies that $a\longmapsto\mathcal{L}_e(a)$ is symmetric.
Now, applying inequalities \eqref{55} for the convex function $f(x)=e^x$ with $a_j=\delta_{ij}\ln a_i$, $1\le i,j\le n+1$, where $\delta_{ij}$ refers to the Kronecker symbol, we get
$$a_i^{\lambda_i}\le \mathcal{L}_\lambda(a)\le \lambda_ia_i+\sum_{\underset{j\neq i}{j=1}}^{n+1}\lambda_j.$$
This, with the fact that $\mathcal{L}_\lambda(1,...,1)=1$, implies that
\begin{gather}\label{42.1}
\dfrac{a_i^{\lambda_i}-1}{a_i-1}\le \dfrac{\mathcal{L}_\lambda(a)-1}{a_i-1}\le \dfrac{\lambda_ia_i-1+\displaystyle\sum_{\underset{j\neq i}{j=1}}^{n+1}\lambda_j}{a_i-1}
\end{gather}
for any $a_i>1$, with reversed inequalities if $a_i<1$. 
Noticing that 
$$\lambda_ia_i-1+\displaystyle\sum_{\underset{j\neq i}{j=1}}^{n+1}\lambda_j=\lambda_i(a_i-1),$$ 
we deduce that 
$$\dfrac{\partial \mathcal{L}_\lambda}{\partial a_i}(1,1, \ldots, 1)=\lambda_i$$ 
for any $i=1,2, \ldots, n+1$. In summary, $\mathcal{L}_\lambda$ is a weighted mean.

Using similar arguments, we prove that $\mathbb{L}_\lambda$ is also a weighted mean.
\end{proof}

The following remark may be of interest.

\begin{remark}
$(i)$ Form \eqref{L1} and \eqref{L2}, it is not hard to see that $\mathcal{L}_\lambda$ and $\mathbb{L}_\lambda$ are homogenous maps of order one.\\
$(ii)$ For $n=1$ and $\lambda_1=\lambda_2=1/2$, \eqref{L1} and \eqref{L2} with simple manipulations yield $\mathcal{L}_{1/2}=\mathbb{L}_{1/2}=L$, where $L$ is the standard logarithmic mean of two variables. This justifies calling $\mathcal{L}_\lambda$ and $\mathbb{L}_\lambda$ weighted logarithmic means.\\
(iii) If $\lambda_i=\frac{1}{n+1}$, $i=1,2,...,n+1$, and $d\mu_{\lambda}(t)=n!\;dt$, \eqref{L1} and \eqref{L2} become, respectively,
$$L_1(a_1,a_2,...,a_{n+1}):=\mathcal{L}_e(a)=n!\;\int_{E_{n}}  \sharp_{t}a\, dt,$$
$$L_2(a_1,a_2,...,a_{n+1}):=\mathbb{L}_e(a)=\Big(n!\;\int_{E_{n}} \big(\nabla_t\;a\big)^{-1}\, dt\Big)^{-1},$$
which are the symmetric multivariate means introduced in the literature, see \cite{NEU} and \cite{RAI}.
\end{remark}

It is worth mentioning that the two weighted logarithmic means ${\mathcal L}_\lambda$ and $\mathbb{L}_\lambda$ are different. Moreover, for $n=1$ they are mutually different from the weighted logarithmic mean defined in \eqref{L}. The following example justifies these claims.
\begin{example}\label{example1}
Let us take $n=2$. We have the following values:
$$
\begin{array}{|c|c|c|c|c|}
		\hline
\lambda_1 &  \lambda_2 & a & \mathcal{L}_\lambda(a) & \mathbb{L}_\lambda(a)\\
\hline
 1/3 & 1/6 & \left(0.5,1,2\right) & 1.19393& 1.19612\\
\hline
0.2 & 0.25 & \left(1.3,1.5,1.9\right) & 1.66722 & 1.66599\\
\hline
\end{array}
$$
$$
\begin{array}{|c|c|c|c|c|}
\hline \lambda_1  &a & \mathcal{L}_\lambda(a) & \mathbb{L}_\lambda(a)& L_\lambda(a) \\
\hline 1 / 3 & (2,1) & 1.60804 & 1.62944 & 1.61423 \\
\hline 0.9 & (4,3)& 3.09329 & 3.08815 & 3.09162 \\
\hline
\end{array}
$$
\end{example}
Other inequalities concerning comparison between the previous weighted logarithmic means and some of the standard means are presented in the following result.

\begin{proposition}\label{Estimate1}
Let ${\bf\lambda}:=(\lambda_1,\lambda_2,...,\lambda_n)\in int(E_n)$ and $a:=(a_1,a_2,...,a_n,a_{n+1})\in (0,\infty)^{n+1}$. Then the following inequalities hold.
\begin{gather}\label{GlA}
\sharp_\lambda a\le {\mathcal L}_\lambda(a)\le \nabla_\lambda a.
\end{gather}
\begin{gather}\label{GLA}
!_\lambda a\le {\mathbb L}_\lambda(a)\le \nabla_\lambda a.
\end{gather}
\end{proposition}
\begin{proof}
Noticing that 
$$\mathcal{L}_\lambda(a)=\int_{E_{n}}\exp\big({\nabla_t \log(a)}\big)d\mu_\lambda(t),$$ 
and applying \eqref{55} with the convex function $x\mapsto e^x$ on $(0,\infty)$, we get \eqref{GlA}.

By applying \eqref{55} for the convex function $x\mapsto x^{-1}$ on $(0,\infty)$, we get \eqref{GLA}.
\end{proof}

\begin{proposition}\label{Estimate2}
For any ${\bf\lambda}:=(\lambda_1,\lambda_2,...,\lambda_n)\in int(E_n)$ and  $a:=(a_1,a_2,...,a_n,a_{n+1})\in (0,\infty)^{n+1}$, there hold
\begin{gather}\label{lL}
!_\lambda a\le {\mathbb L}_\lambda^{-1} (a^{-1})\le {\mathcal L}_\lambda (a),
\end{gather}
where we set $a^{-1}:=(a_1^{-1},..., a_{n+1}^{-1})$.
\end{proposition}
\begin{proof}
We know that, $!_{t}\;a\le \sharp_t\;a\le \nabla_t\;a$,
with $t=(t_1,...t_n)$ and $t_{n+1}=1-\sum_{i=1}^n t_i$.
Multiplying all sides of these inequalities by $d\nu_\lambda(t)$ and then integrating with respect to $(t_1,t_2,...,t_{n})\in E_n$, we obtain
$$\mathbb {L}^{-1}_\lambda(a^{-1})\leq \mathcal{L}_\lambda(a)\leq \nabla_\lambda a.$$
Now, using \eqref{GLA} we obtain
$$\mathbb{L}^{-1}_\lambda(a^{-1})\ge (\nabla_\lambda a^{-1})^{-1}=!_\lambda a,$$
which completes the proof.
\end{proof}

\begin{remark}
According to the values provided in Example \ref{example1}, we have 
$$\mathcal{L}_{(1/3,1/6)}(0.5,1,2)<\mathbb{L}_{(1/3,1/6)}(0.5,1,2),$$
$$\mathcal{L}_{(0.2,0.25)}(1.3,1.5,1.9)>\mathbb{L}_{(0.2,0.25)}(1.3,1.5,1.9).$$
Therefore, $\mathcal{L}_\lambda$ and $\mathbb{L}_\lambda$ are not comparable.
\end{remark}

\subsection{\bf Weighted multivariate identric mean.}

In this subsection, we will present another multivariate mean as well as some of its properties.

\begin{proposition}
Let ${\bf\lambda}:=(\lambda_1,\lambda_2,...,\lambda_n)\in int(E_n)$ and $a:=(a_1,a_2,...,a_n,a_{n+1})\in (0,\infty)^{n+1}$, we set
\begin{equation}\label{I1}
\mathcal{I}_\lambda(a):=\exp\Big(\int_{E_{n}}  \log\left(\nabla_ta\right)\, d \mu_\lambda(t)\Big),
\end{equation}
with $t_{n+1}=1-\sum_{i=1}^{n}t_i$ and $\mu_\lambda$ is the measure defined in \eqref{mu}. Then
${\mathcal I}_\lambda$ is a weighted mean, which we call the multivariate weighted identric mean.
\end{proposition}
\begin{proof}
Conditions $(i), (ii)$ and $(iii)$ of Definition \ref{DP} are straightforward and therefore omitted here. To check the condition $(iv)$, we apply \eqref{55} with the concave function $f(x)=\log x,\; x\in(0,\infty)$ and $a=(1,..,1,a_i,1,...,1)$. We get,
$${\lambda_i}\log(a_i)\le \int_{E_{n}}  \log\left(\nabla_ta\right)\, d \mu_\lambda(t)\le \log\Big(\lambda_ia_i+\sum_{\underset{j\neq i}{j=1}}^{n+1}\lambda_j\Big),$$
and then
$$a_i^{\lambda_i}\le\mathcal{I}_\lambda(a)\le \lambda_ia_i+\sum_{\underset{j\neq i}{j=1}}^{n+1}\lambda_j.$$
This, with the fact that $\mathcal{I}_\lambda(1,...,1)=1$, implies that for any $a_i>1$ we have
\begin{equation}\label{42.1}
\dfrac{a_i^{\lambda_i}-1}{a_i-1}\le \dfrac{\mathcal{I}_\lambda(a)-1}{a_i-1}\le \frac{\lambda_ia_i-1+\displaystyle\sum_{\underset{j\neq i}{j=1}}^{n+1}\lambda_j}{a_i-1},
\end{equation}
with reversed inequalities if $a_i<1$. 
Noticing that 
$$\lambda_ia_i-1+\displaystyle\sum_{\underset{j\neq i}{j=1}}^{n+1}\lambda_j=\lambda_i(a_i-1),$$ 
we find
$$\dfrac{\partial \mathcal{I}_\lambda}{\partial a_i}(1,1, \ldots, 1)=\lambda_i,\, \text{ for every } i=1,2, \ldots, n+1.$$
So, $\mathcal{I}_\lambda$ is a weighted mean.
\end{proof}

The following result provides a comparison between ${\mathcal I}_\lambda$ and some of the standard weighted means. It also states that the two multivariate weighted means $\mathbb{L}_\lambda$ and ${\mathcal I}_\lambda$ are comparable.

\begin{proposition}\label{Estimate}
Let ${\bf\lambda}:=(\lambda_1,\lambda_2,...,\lambda_n)\in int(E_n)$ and $a:=(a_1,a_2,...,a_n,a_{n+1})\in (0,\infty)^{n+1}$. We have the following inequalities
\begin{gather}\label{GIA}
\sharp_\lambda a\le {\mathcal I}_\lambda(a)\le \nabla_\lambda a,
\end{gather}
\begin{gather}\label{LI}
\mathbb{L}_\lambda(a)\le {\mathcal I}_\lambda(a).
\end{gather}
\end{proposition}
\begin{proof}
Applying \eqref{55} for the concave function $x\mapsto \log x$ on $(0,\infty)$, we get
$$\nabla_\lambda \log(a)\le \int_{E_n}\log(\nabla_ta)d\mu_\lambda(t)\le \log(\nabla_\lambda a), $$
which, with the fact that $x\mapsto e^x$ is increasing on $\mathbb{R}$, leads to \eqref{GIA}.

On the other hand, by the use of Jensen integral inequality with the convex function $x\mapsto -\log x$ on $(0,\infty)$, we obtain
$$-\log\int_{E_n}\big(\nabla_t a\big)^{-1}d\nu_\lambda(t)\leq\int_{E_n}\log\Big(\nabla_t a\Big)d\nu_\lambda(t),$$
which, with the fact that $x\mapsto e^x$ is increasing on $\mathbb{R}$, implies \eqref{LI}.
\end{proof}

Finally, we state the following remark.
 
\begin{remark}
(i) For $n=1$, $\mathcal{I}_{1/2}$ coincides with the standard bivariate identric mean $I$.\\
(ii) For $n=1$, ${\mathcal I}_\lambda$ and $I_{\lambda}$ defined by \eqref{I} are not comparable as shown in the following example. 
$$\begin{array}{|c|c|c|c|}
\hline \lambda_1  &a & \mathcal{I}_\lambda(a)& I_\lambda(a) \\
\hline 3/4 & (3,1) & 1.40952 & 1.43367  \\
\hline 0.2 & (6.5,6)& 6.39950 & 6.39893\\
\hline
\end{array}$$
(iii) The weighted means ${\mathcal L}_\lambda$ and ${\mathcal I}_\lambda$ are not comparable. Indeed, taking $n = 2$, we get
$$\begin{array}{|c|c|c|c|c|}
\hline \lambda_1  &\lambda_2 & a& \mathcal{L}_\lambda(a)&\mathcal{I}_\lambda(a)\\
\hline 1/3 & 1/6 & (0.5,1,2) & 1.19393&1.26771  \\
\hline 0.05 & 0.2& (19,1,1) &1.36040 &1.35253\\
\hline
\end{array}$$
So, 
$$\mathcal{L}_{(\frac{1}{3},\frac{1}{6})}(0.5,1,2)\le \mathcal{I}_{(\frac{1}{3},\frac{1}{6})}(0.5,1,2)\text{ and }  \mathcal{L}_{(0.05,0.2)}(19,1,1)\ge \mathcal{I}_{(0.05,0.25)}(19,1,1),$$
\end{remark}

\bibliographystyle{amsplain}

\end{document}